\documentclass[10pt, article]{amsart}

\usepackage{tikz}
\usetikzlibrary{calc}

\usepackage{ae} 
\usepackage[T1]{fontenc}
\usepackage[cp1250]{inputenc}
\usepackage{amsmath}
\usepackage{amssymb, amsfonts,amscd,verbatim}
\usepackage{hyperref}

\usepackage{indentfirst}
\usepackage{latexsym}
\input xy
\xyoption{all}

\usepackage{amsmath}    

\theoremstyle{plain}
\newtheorem{proposition}{Proposition}[section]
\newtheorem{theorem}[proposition]{Theorem}
\newtheorem{corollary}[proposition]{Corollary}

\newtheorem{lemma}[proposition]{Lemma}

\theoremstyle{definition}
\newtheorem{definition}[proposition]{Definition}

\theoremstyle{remark}
\newtheorem{remark}[proposition]{Remark}

\newcommand{\diam}{\operatorname{Diam}\nolimits}

\newcommand{\Rips}{\operatorname{Rips}\nolimits}

\newcommand{\sRips}{\operatorname{sRips}\nolimits}

\def\a{\alpha}
\def\b{\beta}
\def\g{\gamma}
\def\d{\delta}
\def\di{\partial}
\def\e{\varepsilon}
\def\f{\varphi}
\def\l{\lambda}
\def\s{\sigma}

\def\r{\rho}
\def\o{\omega}

\def\cN{\overline{N}}
\def\la{\langle}
\def\ra{\rangle}

\def\ZZ{{\mathbb Z}}
\def\NN{{\mathbb N}}
\def\FF{{\mathbb F}}

\def\f{{\varphi}}

\numberwithin{equation}{section}
\title[Persistent Homology with Selective Rips complexes...]
{Persistent Homology with Selective Rips complexes detects geodesic circles}

\author{\v Ziga ~Virk}
\address{University of Ljubljana, and institute IMFM, Ljubljana, Slovenia}
\email{ziga.virk@fri.uni-lj.si}
\thanks{Research was  supported by Slovenian Research Agency grants No. J1-4001  and P1-0292.}

\begin{document}

\maketitle
\begin{center}
\today
\end{center}

\begin{abstract}
This paper introduces a method to detect each geometrically significant loop that is a geodesic circle (an isometric embedding of $S^1$) and a bottleneck loop (meaning that each of its perturbations increases the length) in a geodesic space using persistent homology. Under fairly mild conditions we show that such a loop either terminates a $1$-dimensional homology class or gives rise to a $2$-dimensional homology class in persistent homology. The main tool in this detection technique are selective Rips complexes, new custom made complexes that function as an appropriate combinatorial lens for persistent homology in order to detect the above mentioned loops.
The main argument is based on a new concept of a local winding number, which turns out to be an invariant of certain homology classes.
\end{abstract}

\section{Introduction}

Homology as a classical invariant is well understood to measure holes in spaces. Its parameterized version persistent homology (PH) on the other hand is thought to also contain geometric information about the underlying space when arising from a filtration built upon it via Rips or \v Cech complexes.  Despite being a focus of intense study from theoretical and practical point of view in the past two decades, the precise nature of geometric information that PH encodes has mostly eluded our understanding. A groundbreaking result in this setting was the discovery of persistence of $S^1$ in \cite{AA}, which demonstrated that lower-dimensional geometric features may induce higher-dimensional algebraic elements in PH. In this spirit a question of geometric information encoded in PH can be recast in the following way: Which geometric properties does PH encode and how does it encode them? Conversely, how to interpret elements of PH in terms of geometric properties?

Some of the few settings in which such an interplay has been theoretically explained contain $1$-dimensional PH of metric graphs \cite{7A},  $1$-dimensional PH and persistent fundamental group of geodesic spaces \cite{ZV, ZV1}, the complete persistence of $S^1$ \cite{AA}; parts of PH of ellipses \cite{Ad5}, regular polygons \cite{Ad4}, and certain spheres \cite{ZVCounterex}.

\begin{figure}[htbp]
\begin{center}
\includegraphics{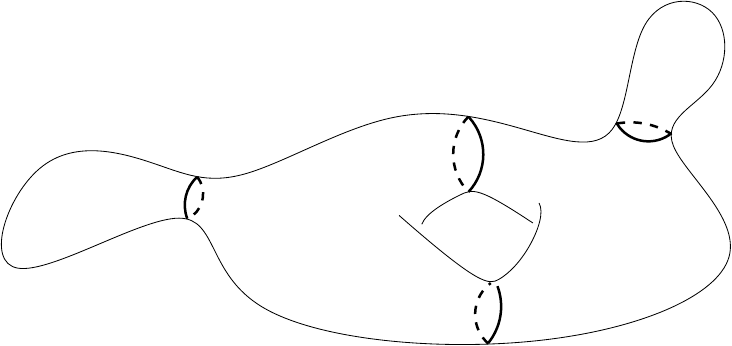}
\caption{A stretched topological torus of dimension $2$. Four of the $\mathcal{GB}$-loops are indicated by bold lines. There is another $\mathcal{GB}$-loop tracing the inner radius of the hole in the middle. All $\mathcal{GB}$-loops can be detected by selective Rips filtration using the results of this paper employing one- and two-dimensional persistent homology. On the other hand, some of these loops may not be detected using standard Rips complexes.}
\label{Fig4}
\end{center}
\end{figure}

In this paper we a \textbf{focus on detection of those loops in geodesic spaces}, that are geodesic circles (isometric embeddings of $S^1$ equipped with a geodesic metric) and bottleneck loops (i.e., each of its perturbations increases the length), see Figure \ref{Fig4} for an example. Let us call them $\mathcal{GB}$-loops. These form a particular class of geodesics in the sense of classical differential geometry hence detecting them we are detecting parts of the length spectrum \cite{H1, H2}. $\mathcal{GB}$-loops are obvious significant geometric features of a space.  All lexicographically minimal homology and fundamental group bases \cite{ZV} consist of $\mathcal{GB}$-loops (in the case of fundamental groups these have to be connected to a basepoint by a path). Each systole \cite{Katz} is a $\mathcal{GB}$-loop and corresponds to the first critical value of $1$-dimensional persistent fundamental group \cite{ZV}.    In \cite{ZV2} and \cite{ZV} it was shown how a $\mathcal{GB}$-loop may sometimes be detected using PH. If $\a$ is a geodesic circle of circumference $|\a|$ in a geodesic space $X$ then in the case of Rips filtration there are three ways of detecting it:
\begin{enumerate}
 \item \textbf{A topological footprint} appearing  if $\a$ is a member of a lexicographically shortest homology base: at $|\a|/3$ a $1$-dimensional PH element ceases to exist\cite{ZV}.
 \item \textbf{A combinatorial footprint} arising from the internal combinatorics of the Rips complex of a circle as described in \cite{AA}:  under certain local conditions a part of PH of $\a$ (i.e., odd-dimensional elements) may appear in PH of $X$ \cite{ZV2};
 \item \textbf{A geometric footprint} appearing under certain local geometric conditions at $\a$ in the absence of a topological footprint: a corresponding $2$-dimensional element of PH of $X$ appears at $|\a|/3$ \cite{ZV2}.
\end{enumerate}
These results provide ways to detect lengths of $\a$ but often turn out to be too restrictive in their conditions. If we cut off a geodesic sphere above the equator and consider the lower remaining part, its boundary (now a $\mathcal{GB}$-loop) can be detected in the above mentioned way only if the cut is made sufficiently far from the equator. If the cut is too low, the  mentioned methods can't detect $\a$, see the example in Subsection \ref{SubSectExample}.

The aim of this paper is to expand the method of a geometric footprint in order to detect more $\mathcal{GB}$-loops under simpler and more general conditions. The first step is to consider the choice of complexes (discretization) used in the construction of PH. Various constructions are currently being used, ranging from classical complexes (Rips, \v Cech), complexes aimed at computational convenience (alpha, wrap, witness) or constructions based on geometric intuition (metric thickening) \cite{Ad1, Ad2, Ad3}. In Definition \ref{DefSRips} we present \textbf{selective Rips complexes} as custom made complexes specifically designed to detect local features. They can be thought of as a combinatorial lens for persistent homology. The idea is that for small dimensions (up to dimension $2$ in our case) a selective Rips complex coincides with the usual Rips complex, while higher-dimensional simplices are forced to be ``thin'' in all but two directions. The reconstruction properties of a more general version of selective Rips complexes are treated in the follow-up paper \cite{Lem}.

As the \textbf{main result} (Theorem \ref{ThmMain1} and Corollaries \ref{Cor5} and \ref{Cor6}) we prove that the length of each $\mathcal{GB}$-loop $\a$ admitting an arbitrarily small geodesically convex neighborhood can be detected by persistent homology via appropriate selective Rips complexes. In particular, such $\a$ generates a $2$-dimensional or terminates a $1$-dimensional homology class at $|\a|/3$. This is a significant generalization  when compared to the corresponding result in  \cite{ZV2}, which required generous assumptions on the geometry and size of neighborhoods of $\a$. A quantitavtive analysis of Remark \ref{Rem:brilliant} also allows us to deduce specific conditions, under which PH via classical Rips complexes detect $\mathcal{GB}$-loops. As our main tool we develop local winding number at $\a$ (see Definition \ref{DefWNHomology}) and prove they are homology invariant in an appropriate setting. We conclude with a discussion on the localization of $\a$ and an example.

The structure of the paper is the following. In Section 2 we present preliminaries on the setting and explain local geometry of bottleneck loops. In Section 3 we introduce local winding numbers. In Section 4 we define selective Rips complexes  and prove that local winding numbers are homology invariants in their setting. In Section 5 we prove our main results, with the localization discussion and an example concluding the paper in Subsections 5.1 and 5.2 respectively.

\section{Preliminaries}
\label{SecPrelims}

In this section we recall the context of geodesic circles in geodesic spaces and provide some simple lemmas clarifying the context.

Given $x$ in a metric space $X$ (or $A \subset X$) and $r>0$, notation $N(x,r)$ (or $N(A,r)$ respectively) denotes the open $r$-neighborhood of $x$ (of $A$ respectively). Notation $\cN(x,r)$ (or $\cN(A,r)$ respectively) will denote the closed $r$-neighborhood. A path (or a loop) in $X$ is a continuous map from an interval (or $S^1$) to $X$. Its length is denoted by $|\a|$. Throughout the paper a path (or a loop) may be thought of as a map or as the image of the corresponding map. The use of double interpretation is often geometrically more convenient and simplifies descriptions, while it shouldn't effect the clarity of the presentation. A choice of orientation will be specifically mentioned when required for the sake of precision, for example when talking about the winding number. On other occasions we will not mention orientation. For example, when talking about $\a$ being the unique shortest loop in some set, it is apparent that reparameterizations of $\a$ and $\a^-$ are also shortest loops as maps, which encourages us to think of $\a$ as a subset in this setting to describe the geometry more elegantly.

A  metric space $X$ is \textbf{geodesic}, if for each $x,y\in X$ there exists a path between them of length $d(x,y)$, i.e., an isometric embedding $\gamma \colon [0, d(x,y)]\to X$ satisfying $\g(0)=x$ and $\g(d(x,y))=y$. Such a path will be called a \textbf{geodesic} (a terminology which may differ from some usages in differential geometry). A subset $A \subseteq X$ is geodesically convex, if $\forall x,y, \in A$, all the geodesics (in $X$) from $x$ to $y$ lie in $A$. The concatenation of paths $\a, \b$ in $X$ is defined if the endpoint of $\a$ coincides with the initial point of $\b$ and is denoted by $\a * \b$. The converse path of $\a\colon [a,b]\to X$ will be denoted by $\a^-\colon [a,b]\to X$, i.e., $\a^-(t)=\a(a+b-t)$. As each loop is also a path, the terminology can also be used for loops. For example, if $\a$ is a loop, $\a*\a$ is well defined up to homotopy by choosing any point of $\a$ as the initial point of the loop $\a$. For $k\in \ZZ$  
the $k$-fold concatenation of an oriented loop $\a$ is defined as:
\begin{itemize}
\item the constant loop if $k=0$,
\item the $k$-fold concatenation of $\a$, (i.e.,  $\a*\a*\ldots*\a$ with $k$ terms) if $k>0$,
\item the $|k|$-fold concatenation of $\a^-$, if $k<0$.
\end{itemize}
 Given a triple of points $x_1, x_2,x_3\in X$, its \textbf{filling} is a loop in $X$ obtained by concatenating (potentially non-unique) geodesics from $x_1$ to $x_2$, from $x_2$ to $x_3$ and from $x_3$ to $x_1$. Note the the length of a filling is $d(x_1,x_2)+ d(x_2, x_3)+ d(x_3, x_1)$.

For $c>0$ let $S^1_c$ be a circle equipped with the geodesic metric that results in the circumference $c$.  A \textbf{geodesic circle} in $X$ is an isometrically embedded $S^1_c$ for some $c>0$. A loop $\a$ of finite length in $X$ is a \textbf{bottleneck loop}, if it is the shortest representative of the (free unoriented) homotopy class of $\a$ (considered as a map) in some neighborhood of $\a$ (considered as a subset of $X$). For example, the equator on the standard unit sphere is a geodesic circle, which is not a bottleneck loop. The loop on $x^2+y^2-z^2=1$ determined by $z=0$ is a bottleneck loop. It is not difficult to find bottleneck loops which are not geodesic circles.

Space $X$ is \textbf{semi-locally simply connected}  if  $\forall x\in X, \forall \e > 0, \exists \d>0$ so that each loop in $N(x,\d)$ is contractible in $N(x, \e)$. Manifolds and simplicial complexes are semi-locally simply connected.

\begin{lemma}
 \label{LemmaA}
 Assume $X$ is a geodesic semi-locally simply connected space. For each $\a \colon S^1 \to X$ there exists $\r>0$ such that if $\b\colon S^1 \to X$ satisfies  $d(\a(t),\b(t))< \r, \forall t\in S^1$, then $\a \simeq \b$.
\end{lemma}

\begin{proof}
 Choose $\r>0$ such that for each $x\in \a(S^1)$ each loop in $\cN(x,\r)$ is contractible in $X$. Define $\mu = \max_{t\in S^1} \{d(\a(t),\b(t))\} < \r$. Partition $S^1$ into small intervals so that the image of each of such an interval is of diameter at most $\r-\mu$ when mapped by $\a$ or $\b$. Also, connect the corresponding endpoints of such intervals by geodesics, these are of length at most $\mu$. Such partition decomposes the ``difference'' between $\a$ and $\b$ into loops of diameter at most $ \r$, see Figure \ref{Fig1} for a sketch. Such loops are contractible by our assumption, hence $\a \simeq \b$.
 
 \begin{figure}[htbp]
\begin{center}
\includegraphics{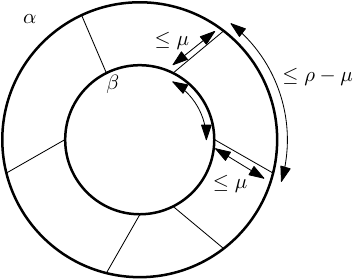}
\caption{A sketch of a decomposition of Lemma \ref{LemmaA}. The images of radial arcs are of length at most $\mu$ while the images of angular are of diameter at most $\r-\mu$.}
\label{Fig1}
\end{center}
\end{figure}
\end{proof}

\begin{lemma}
 \label{LemmaB}
Assume $X$ is a geodesic semi-locally simply connected space. For each geodesic circle $\a \colon S^1_{|\a|} \to X$ there exists $\r>0$ such that if $\beta\colon S^1 \to N(\a,\r)$, then $\beta$ is homotopic to a $k$-fold concatenation of $\a$ for some $k\in \ZZ$.
\end{lemma}

\begin{proof}
Choose $\r>0$ such that $4 \r < |\a|$ and for each $x\in \a(S^1)$ each loop in $\cN(x,3 \r)$ is contractible in $X$. Define $\mu = \max_{t\in S^1} \{d(\b(t), \a(S^1))\}$ and note that $\mu < \r$. Let $\Pi$ be a partition of  $S^1$ into small intervals so that the image of each interval is of diameter at most $\r-\mu$ when mapped by $\b$. For each point $t\in S^1$ appearing as an endpoint of two intervals of $\Pi$ choose a point $s_t\in S^1_{|\a|}$ such that $d(\beta(t), \alpha(s_t)) \leq \mu$. For each interval of $\Pi$ with endpoints $t, t'$:
\begin{enumerate}
\item note that $d(\alpha(s_t),\alpha(s_{t'}))< \rho + \mu < |\a|/2$;
\item connect $\alpha(s_t)$ to $\alpha(s_{t'})$ by the geodesic $\g_{t, t'}$ along $\a$, which exists as $\a$ is a geodesic circle;
\item connect $s_t$ and  $s_{t'}$ by the geodesic in $S^1_{|\a|}$, whose image is the geodesic $\g_{t, t'}$.
\end{enumerate}
This induces a map $\f \colon S^1 \to S^1_{|\a|}$, mapping $S^1$ as the domain of $\b$ along with partition $\Pi$ to $S^1_{|\a|}$ as the domain of $\a$, so that each interval of $\Pi$ is mapped bijectively onto a geodesic in $S^1_{|\a|}$ between $\f(t)=s_t$ and $ \f(t')=s_{t'}$, which is of length at most $\rho + \mu< |\a|/2$.  Observe that $\b \simeq \a \circ \f$ as, similarly as in Lemma \ref{LemmaA}, the ``difference'' between them can be decomposed into loops of diameter at most $3 \r$ using partition $\Pi$ as shown on Figure \ref{Fig2}. This concludes the proof as $\a \circ \f$ is homotopic to a $k$-fold concatenation of $\a$ for some $k\in \ZZ$.
  \begin{figure}[htbp]
\begin{center}
\includegraphics{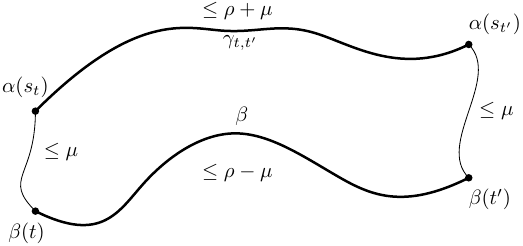}
\caption{An excerpt of homotopy between $\b$ and $\f \circ \a$ of Lemma \ref{LemmaB}. The inequalities indicate the bound on the diameter of each part. Taking into the account $\mu < \rho$,  the entire loop is of diameter at most $3 \rho$, hence it is contractible. The corresponding nullhomotopy provides a homotopy between a part of $\beta$ on an interval of $\Pi$ (below) and the corresponding part $\g_{t,t'}$ of $\f \circ \a$ (above).}
\label{Fig2}
\end{center}
\end{figure}
\end{proof}

\begin{lemma}
 \label{LemmaD}
Assume $X$ is a geodesic semi-locally simply connected space. For each geodesic circle $\a \colon S^1_{|\a|} \to X$ that is also a bottleneck loop,  there exists $\r>0$ such that $\a$ is the shortest non-contractible loop in $N(\a, \r)$.
\end{lemma}

\begin{proof}
 Choose $\rho>0$ small enough and $n\in \{5, 6, \ldots\}$  large enough so that  for each $x\in \a(S^1)$ each loop in $\cN \big(x,\frac{|\a|}{n} + 2 \r \big)$ is contractible in $X$.
Further decrease $\r$ so that the following are satisfied:
\begin{itemize}
 	 \item $\rho$ satisfies Lemma \ref{LemmaB};
	 \item $2n \r < |\a|$;
	 \item $|\a|/n + 2 \r < |\a|/2$;
	 \item $\a$ is the shortest representative of its homotopy class 	 in $N(\a, \r)$.
\end{itemize}
 Assume there exists a shorter loop $\b$ in $N(\a, \r)$. Partition $S^1$ (as the domain of $\b$) into $n$ intervals so that the image of each of the intervals is of length less than $|\a|/n$ when mapped by $\b$. 
 For each point $t\in S^1$ appearing as an endpoint of two of the mentioned intervals above  choose a point $s_t\in S^1_{|\a|}$ such that $d(\b(t), \a(s_t)) \leq \mu$. As in Lemma  \ref{LemmaB}, for each of the mentioned intervals with endpoints $t, t'$ we connect  $\alpha(s_t)$ to $\alpha(s_{t'})$ by the unique geodesic $\g_{t, t'}$ of length less than $|\a|/n + 2 \r < |\a|/2$ along $\a$. Concatenating segments of the form $\g_{t, t'}$ we thus obtain a loop along $\a$ of length less than $n(|\a|/n + 2 \r)=|\a| + 2 n \r < 2|\a|$, which is homotopic to $\b$ by the same argument as provided in Lemma \ref{LemmaB}. As it is of length less than $2\a$, this loop is either  homotopic to $\a$ up to orientation or contractible. The first of these conclusions contradicts the assumption that $\a$ is a bottleneck loop (as $\b$ would in this case be shorter that $\a$ and homotopic to it). Hence $\b$  is contractible and the lemma is proved.
\end{proof}

The following is a standard result that can be proved directly or using the Arzel\` a-Ascoli Theorem. We state it here for completeness. 

\begin{lemma}
 \label{LemmaC}
Assume $X$ is a compact geodesic space and $f_i\colon I \to X$ is a sequence of geodesics in $X$. Then there exists a subsequence of $f_i$ converging point-wise to a (potentially one-point) geodesic in $X$.
\end{lemma}
\section{Geometric Lemmas}

In this section we present a sequence of geometric lemmas in our setting. These results are the fundamental reason why the selective Rips complexes as defined in the next section detect locally shortest geodesic circles through persistent homology.

\textbf{Overall assumptions and declarations for this section}: 
\begin{enumerate}
	 \item assume $X$ is a geodesic, locally compact,  semi-locally simply connected space;
	 \item assume loop $\a$ in $X$ is a geodesic circle and a bottleneck loop;
	 \item assume loop $\a$ has an arbitrarily small geodesically convex closed neighborhood;
	 \item By the previous assumptions and Lemma \ref{LemmaD} there exists a compact geodesically convex neighborhood $N$ of $\a$ in which $\a$ is the shortest non-contractible loop;
	 \item Furthermore, by Lemma \ref{LemmaB} we may assume that each loop in $N$ is homotopic to the $k$-fold concatenation of $\a$ for some $k\in \ZZ$;
	 \item Choose $T>0$ such that $N \supset \cN(\a, T)$ and $d(\cN(\a, T), N^C)>0$.
\end{enumerate}

\begin{definition}
\label{DefCicrum}
 A triple of points $x_1, x_2, x_3\in N$ is \textbf{circumventing} if any of its fillings in $N$ is homotopic to $\a$ (or possibly $\a^-$, if an orientation is taken into account).
\end{definition}

\begin{remark}
\label{Rem1}
 Note that the diameter of a circumventing triple in $N$ is at least $|\a|/3$. If the diameter of the triple is precisely $|\a|/3$ the points lie on $\a$. If the diameter of the triple is less than $|\a|/2$, then the ``difference'' between various fillings of the triple is generated by loops of length less than $|\a|$ (the argument is similar to the ones in the lemmas of Section \ref{SecPrelims} or in \cite[Proposition 4.6]{ZV}). As such loops in $N$ are contractible, the homotopy type of a filling is well defined in this case. 
\end{remark}

\begin{lemma}
\label{Lemma1}
For each $t>0$ there exists $\mu_t> |\a|/3$ such that each circumventing triple in $N$ of diameter less than $\mu_t$ has a filling that is contained in $\cN(\a, t).$ 
\end{lemma}

\begin{proof}
Assume there exists $t>0$ such that for each $\mu>|\a|/3$ there exists a circumventing triple $x_\mu, y_\mu, z_\mu$ of diameter at most $\mu$, whose filling is not a subset of $\cN(\a, t)$. Choose a decreasing sequence $a_i$ converging to $|\a|/3$ so that:
\begin{itemize}
 \item the sequences $x_{a_i}, y_{a_i}, z_{a_i}$ converge in $N$;
 \item the sequence of fillings converges to a filling in $N$  (use Lemma \ref{LemmaC}).
\end{itemize}
By Lemma \ref{LemmaA} the limiting filling is homotopic to $\a$ and of length $|\a|$, hence it coincides with (a reparameterization of) $\a$. Consequently, almost all fillings of triples $x_{a_i}, y_{a_i}, z_{a_i}$ must be contained in $\cN(\a, t)$, a contradiction.
\end{proof}

\begin{definition}
 \label{DefWind}
 Choose an orientation on $\a$. The \textbf{winding number} $\o(\b)$ of an oriented loop $\b$ in $N$ equals $k\in \ZZ$ if $\b$ is homotopic to the $k$-fold concatenation of $\a$.
 Assume the diameter of an oriented triple (we can think of it as an oriented $2$-simplex) $\la x_1, x_2, x_3 \ra$ in  $ N$ is less than $|\a|/2$. The \textbf{winding number} $\o(x_1, x_2, x_3)$ equals $k\in \ZZ$ if a filling of $x_1, x_2, x_3$ is homotopic to the $k$-fold concatenation of $\a$.
\end{definition}

By Remark \ref{Rem1} the winding number of a triple is well defined.

\begin{lemma}
 \label{LemmaDa}
 There exists $D=D_\a >  |\a|$ such that each non-contractible loop in $N$  of length less than $D$ is homotopic to $\a$ (or possibly $\a^-$, if an orientation is taken into account).
 \end{lemma}

\begin{proof}
The proof is similar to that of Lemma \ref{Lemma1} with a direct application of the Arzel\` a-Ascoli Theorem. Assume that for each $s> |\a|$ there exists a non-contractible loop $\g_s$ in $N$ of length less than $s$, which is not homotopic to $\a$. The bound on the length provides the equicontinuity condition and we can use the Arzel\` a-Ascoli Theorem to assume there exists a decreasing sequence $s_i$ converging to $|\a|$ so that $\g_{s_i}$ pointwise converge to a loop $\g$. As all $\g_{s_i}$ are of length at least $|\alpha|$, so is $\gamma$. As for any $i$, each member of the collection $\{\g_{s_i}\}_{i  \geq i'}$ is of length less than $s_{i'}$, the same holds for $\gamma$. Thus, $\gamma$ is of length $|\alpha|$.

By Lemma \ref{LemmaA} $\g$ is homotopic to almost all $\g_{s_i}$. On the other hand, $\g$ is either $\a$ (or possibly $\a^-$ if the orientation is taken into account) or contractible, which is a contradiction with our assumption. 
\end{proof}
 
\begin{corollary}
 \label{Cor2}
 Each triple of points in $N$ of diameter less than $D/3$ is of winding number either $0,1,$ or $-1$.
\end{corollary}

\begin{lemma}
 \label{Lemma2.5}
 Assume, $\b$ and  $\g$ are two loops in $N$ with a common point $x_0$, and let $\b * \g$ denote their concatenation at that point. Then $\o(\b * \g)= \o(\b)+ \o(\g)$.
\end{lemma}

\begin{proof}
 If $\b$ is homotopic to $k \a$ and $\g$ is homotopic to $k' \a$ then $\b * \g$ is homotopic to $(k \a) * \d * (k' \a) * \d^-$ with $\d$ being a loop from a point on $\a$ to $x_0$ (traced by $x_0$ via  the reversed homotopy of $\b$) and back to $\a$  (traced by  $x_0$ via the homotopy of $\g$). As the fundamental group of $N$ is Abelian by Lemma \ref{LemmaB} $(k \a) * \d * (k' \a) * \d^- \simeq (k \a) * (k' \a) \simeq (k+k')\a$.
\end{proof}

Recall that an oriented quadruple ($3$-simplex) induces four oriented triples ($2$-simplices) via the boundary map. 

\begin{lemma}
 \label{Lemma3} 
 Assume $\s= \la x_1, x_2, x_3, x_4 \ra$ is an oriented quadruple in $N$ of diameter less than $\min \{\mu_T, D/3\}$ and fix an orientation on $\a$. Then either none, two, or four of the induced oriented triples are circumventing and their winding numbers add up to $0$.
\end{lemma}

\begin{proof}
 Assume $L_1, L_2, L_3$ and $L_4$ are oriented fillings of the induced oriented triples $\la x_2, x_3, x_4 \ra ,\la x_3, x_1, x_4 \ra, \la x_1, x_2, x_4\ra,  \la x_2, x_1, x_3\ra$. Choosing $x_4$ as the common point to perform a concatenation, note that $L_1 * L_2 * L_3 \simeq L_4^-$ hence by Lemma \ref{Lemma2.5} the winding numbers add up to $0$. The proof is concluded by observing that the winding numbers can be either $0, 1$ or $-1$ by Lemma \ref{LemmaDa}.
 \begin{figure}[htbp]
\begin{center}
\includegraphics{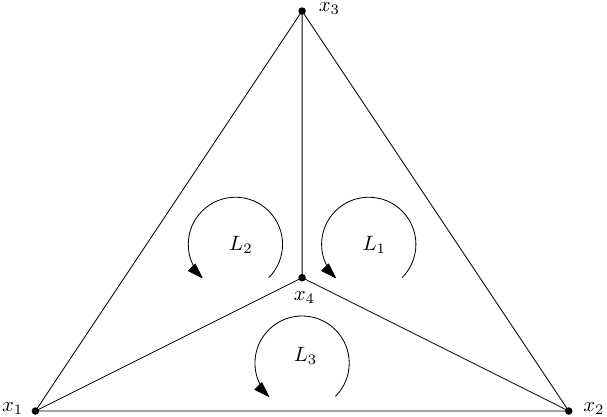}
\caption{A sketch of a decomposition of Lemma \ref{Lemma3}.}
\label{Fig3}
\end{center}
\end{figure}
\end{proof}

\section{Selective Rips complexes}

In this section we define selective Rips complexes and the winding number of a $2$-dimensional homology class in appropriate selective Rips complexes.

For the sake of clarity we first recall a definition of (open) Rips complexes that will be used here. 
Given a scale $r>0$, the \textbf{Rips complex} $\Rips(X; r)$ is an abstract simplicial complex with the vertex set $X$ defined by the following rule: a finite $\s \subset X$ is a simplex iff $\diam(\s)<r$.

\begin{definition}
 \label{DefSRips}
Let $Y$ be a metric space, $r_1 \geq r_2, n\in \NN$.  \textbf{Selective Rips complex} $\sRips(Y; r_1, n, r_2)$ is an abstract simplicial complex defined by the following rule: a finite subset $\s\subset Y$ is a simplex iff the following two conditions hold:
\begin{enumerate}
 \item $\diam (\s) < r_1$;
 \item there exist subsets $U_0, U_1, \ldots,  U_n\subset U$ of diameter less than $r_2$ such that $\s \subset U_0 \cup U_1 \cup \ldots \cup U_n$.
\end{enumerate} 
\end{definition} 

Condition (1) of Definition \ref{DefSRips} implies that $\sRips(Y; r_1, n, r_2)$ is a subcomplex of $\Rips(Y; r_1)$. Condition (2) of Definition \ref{DefSRips} implies that we should be able to partition (cluster) vertices of $\s$ into $n+1$ clusters of diameter less than $r_2$.

\begin{remark}
 Given a subset of a metric space, there are two natural notions of its size. One is the radius of its smallest enclosing ball (or the infimum of radii, if no single ball is the smallest). The second one is its diameter. Combined with the nerve construction, sets of prescribed ``size'' in the first case induce \v Cech complexes and in the second case Rips complexes. See \cite{ZV3} for a detailed exposition of this story.
 
Definition \ref{DefSRips} combines a global (1) and local (2) bound on the size of groups of vertices of a simplex $\s$. In this case, both sizes are expressed in terms of a diameter. We might as well change condition (2) into a version where the size in expressed in terms of the radius of the smallest enclosing ball in the following way:

\emph{there exist $x_0, x_1, \ldots, x_n\in \sigma$  (or, alternatively, $x_0, x_1, \ldots, x_n\in X$) such that $\s \subset N(x_0, r_2) \cup N(x_1, r_2) \cup \ldots \cup N(x_n, r_2)$.}

The results of this paper also hold for this global-Rips local-\v Cech construction with the same arguments albeit slightly different parameters. Similar arguments could also be devised for selective versions of \v Cech complexes: global-\v Cech local-Rips, and global-\v Cech local-\v Cech.
\end{remark}

Throughout this section we will maintain the notational definitions established in the previous section. We additionally make the following \textbf{declarations}:
\begin{itemize}
 \item define $\nu = d(\cN(\a, T), N^C)>0$;
 \item choose $a\in (|\a|/3, \min\{D_\a /3, \mu_T/3\})$ and $b \in (0, a]$;
 \item  choose an Abelian group $G$.
\end{itemize}

In particular, given a neighborhood $N$ (due to Lemma \ref{LemmaD}) of $\alpha$, choose $D_\alpha$ according to Lemma \ref{LemmaDa}, choose $T>0$ and define the corresponding $\nu$, choose $\mu_T$ according to Lemma \ref{Lemma1}, and pick $a$ and $b$ within the ranges of the stated inequalities.

\begin{definition}
 \label{DefWNHomology}
Let $c=\sum_i \l_i \s_i$ be a $2$-cycle in the chain group $C_2(\sRips(X; r_1, 2, r_2);G)$ with $\l_i\in G$ and $\s_i \in\sRips(X; r_1, 2, r_2)$ being represented as oriented triples $\forall i $. The \textbf{winding number} of $c$ in $N$ is defined as
 $$
 \o_N(c)= \sum_{\s_i \subset N} \l_i \o_N(\s_i)\in G.
 $$
\end{definition}

\begin{proposition}
\label{PropMain1}
Assume that either $a < \nu$ or $b < \min \{|\alpha|-2a, \nu\}$, with the later condition implicitly requiring that $|\alpha|-2a > 0$.
Then the winding number of a $2$-cycle $c$ in $N$ is an invariant of the homology class $[c]\in H_2(\sRips(X; a, 2, b);G)$.
\end{proposition}

\begin{proof}
 It suffices to prove that $\o_N(c + \di \s)=\o_N(c)$ for any oriented $3$-simplex in $\sRips(X; a, 2, b)$ or equivalently, that $\o_N(\di \s)=0$. Assume an oriented face $\s'$ of $\di \s$ is contained in $N$ and circumventing. As $a< \mu_T$ Lemma \ref{Lemma1} implies $\s'$ is contained in $\cN(\a,T)$.
 We now claim that all four vertices of $\s$ are contained in $N$:
\begin{itemize}
 \item If $a< \nu$ the claim follows from the definition of $\nu$ as $\diam(\s) < a$. 
 \item If $b < |\alpha|-2a$, then the pairwise distances between the three vertices of $\s'$ are larger than $b$ as their filling is of length at least $|\alpha|$. Hence the distance of the fourth vertex of $\sigma$ to one of the vertices of $\sigma'$ is less than $b$. If additionally, $b < \nu$, then the claim holds by the definitions of $\nu$ and $\sRips(X; a, 2, b)$. 
\end{itemize} 
 By Lemma \ref{Lemma3} the sum of the winding numbers of the four $2$-simplices of $\di \s$ equals $0$ hence $\o_N(\di \s)=0$.
 \end{proof}

It is easy to see that the winding number in this setting determines a homomorphism $H_2(\sRips(X; a, 2, b);G)\to G$.

\section{Detection of geodesic circles}
 In this section we combine the results of previous sections to prove how selective Rips complexes may detect geodesic circles, which are also bottleneck loops in geodesic spaces. The main technical result states that such loops may induce a non-trivial two-dimensional homology class in persistent homology. 
 
 For a loop $\g\colon I \to X$ ($\g(0)=\g(1)$) and $r>0$, a $r$-\textbf{sample} of $\g$ is a sequence $\g(t_0), \g(t_1), \alpha(t_2), \ldots, \g(t_k)$, where $t_0=0< t_1 < \ldots < t_k=1$ and for each $i$ $\diam(\g|_{[t_i, t_{i+1}]}) < r$ holds. A  $r$-sample will often be identified by a simplicial loop or a simplicial $1$-chain in $\sRips(X;a,2,b)$ if $a \geq r$.
 
\begin{theorem}
 \label{ThmMain1}
 Let $X$ be a geodesic locally compact, semi-locally simply connected space and let $G$ be an Abelian group. Assume $\a$ is a geodesic circle in $X$ satisfying the following properties:
\begin{enumerate}
 \item $\a$ is a bottleneck loop;
 \item $\a$ is homologous to a non-trivial $G$-combination of loops $\beta_1, \beta_2, \ldots, \beta_k$ of length at most $|a|$, none of which equals $\a$ as a subset of $X$; 
 \item $\a$ has arbitrarily small geodesically convex neighborhoods.
\end{enumerate}
Then there exist bounds $B_1>|\a|/3$ and  $B_2>0$ such that for all increasing bijections $a \geq b\colon (0,\infty)\to (0,\infty)$, and for all $r>0$ such that $B_1>a(r)> |\a|/3$ and $B_2>b(r)$, there exists a non-trivial 
$$
Q_r \in H_2(\sRips(X;a(r),2,b(r));G) 
$$
satisfying the following properties: 
\begin{enumerate}
 \item $\forall r_1<r_2$ with $a(r_i)> |\a|/3$ and $b(r_i)<B, \forall i$, we have $i^G_{r_1, r_2} (Q_{r_1})=Q_{r_2}$, where $i^G_{r_1, r_2} \colon H_2(\sRips(X;a(r_1),2,b(r_1));G)  \to H_2(\sRips(X;a(r_2),2,b(r_2));G) $ is the natural inclusion induced map.
 \item $\forall q: a(q) \leq |\a|/3$ there exists no $Q_q$ with $i_{q,r}(Q_q)=Q_r$.
\end{enumerate}
\end{theorem}

\begin{remark}
\label{Rem:brilliant}
 Bounds $B_1, B_2$ in Theorem \ref{ThmMain1} are not unique and arise from the assumptions of Proposition \ref{PropMain1}. 
 In the proof of Theorem \ref{ThmMain1} the condition $b < \min \{|\alpha|-2a, \nu\}$ is used.
 A small increase in $B_1$ accompanied by an appropriate decrease in $B_2$ will result in a different pair of parameters for which the theorem still holds. This interplay has the following geometric interpretation: the longer the lifespan of $Q_r$ is, the thinner the triangles in selective Rips complexes need to be.
 
The other condition of Proposition \ref{PropMain1}, namely $a < \nu$, can be used to derive a version of Theorem \ref{ThmMain1} for Rips complexes. Following the notation of the proof below, suppose $N=N(\a, \widetilde T) \cong S^1 \times (0,1)$ is a compact geodesically convex neighborhood of $\a$, in which $\a$ is the shortest non-contractible loop. Assume $\widetilde T > |\a|/3$. Then in the proof below we can choose $T\in (0, \widetilde T)$ so that $\nu = \widetilde T-T > |\alpha|/3$, and hence the conditions of Proposition \ref{PropMain1} hold for $a$ slightly above $|\alpha|/3$ and $b=a$, i.e., in the case of Rips complexes. 

\textbf{In particular, if $\widetilde T > |\a|/3$ then the conclusions of Theorem \ref{ThmMain1} and Corollaries \ref{Cor5} and \ref{Cor5} hold for Rips complexes.}
\end{remark}

\begin{proof}[Proof of Theorem \ref{ThmMain1}]
The construction of $Q_r$ appeared in \cite{ZV2}. There exist singular $2$-simplices $\tilde\Delta_{\tilde j}$ in $X$ and $\tilde h_{\tilde j}, g_i\in G$ in $X$ so that the following equality of $1$-chains holds:
$$
[\alpha]_G = \sum_{i=1}^k g_i [\beta_i]_G + \di \sum_{\tilde j=1}^{\tilde k'} \tilde h_{\tilde j} [\tilde\Delta_{\tilde j}]_G.
$$
We further subdivide  singular $2$-simplices $\tilde\Delta_{\tilde j}$ into smaller singular $2$-simplices $\Delta_j$ so that for some $h_j\in G$ the equality
 $$
L_\alpha = \sum_{i=1}^k g_i L_i + \di \sum_{j=1}^{k'} h_j \Delta_j
$$
holds in the second chain group of $\Rips(X,|\alpha|/3)$ with the following declarations and conditions:
\begin{itemize}
 \item the diameter of each singular simplex $\Delta_j$ is less than $|\alpha|/3$;
\item each singular $2$-simplex $\Delta_j$ also represents the $2$-simplex in $\Rips(X,|\alpha|/3)$ determined by its vertices (and the inherited orientation).
 \item $L_\alpha$ and $L_i$ are  $(|\alpha|/3)$-samples of $\alpha$ and $\beta_i$ correspondingly;
 \item each close $(|\alpha|/3)$-sample above contains three points on the corresponding loop, that divide the loop into three parts of the same length. We call these points equidistant points and keep in mind that the ``equidistance'' refers to the length of the loop $\a$ or $\b_i$ between them.
\end{itemize}

With this we have transitioned from singular $1$-chains in $X$ to simplicial $1$-chains in $\sRips(X; a(r),2,b(r); G)$ with $a(r) \geq |\alpha|/3$. Next we use the type of decomposition presented by Figure \ref{FigNullhomotopy}  to express $1$-chains $L_\a$ and $L_i$ as boundaries: 
\begin{itemize}
 \item $L_i=\di \sum_{p=1}^{k_p} \tau_{i,p}$;
 \item $L_\alpha = \di \sum_{l=1}^{k_\alpha}  \sigma_l$.
\end{itemize}
\begin{figure}
\begin{tikzpicture}[scale=1]

\node (q1) at (60:1.9){}; 

\node (q2) at (180: 1.9){};

\node (q3) at (-60:1.9){};

\draw[top color=gray,bottom color=white, fill opacity =1] (60:1.9) -- (180:1.9) -- (-60:1.9) -- cycle;
\foreach \x  in {1,2,3,5,6,7,9,10,11}
	\draw (-\x * 30+90:1.9) -- (-\x * 30+90:2.0) (-\x * 30+60:2.4) node {$x_{\x }$};
\foreach \x  in {0}
	\draw (-\x * 30+90:1.9) -- (-\x * 30+90:2.0) (-\x * 30+60:2.4) node {$x_{\x }$};
\foreach \x  in {4}
	\draw (-\x * 30+90:1.9) -- (-\x * 30+90:2.0) (-\x * 30+60:2.4) node {$x_{\x }$};
\foreach \x  in {8}
	\draw (-\x * 30+90:1.9) -- (-\x * 30+90:2.0) (-\x * 30+60:2.7) node {$x_{\x }$};
\foreach \x  in {0, 1, ..., 11}
	\draw [fill=black] (-\x * 30+90:1.9) circle (.1cm);
\foreach \x  in {0, 1, ..., 11}
	\draw (-\x * 30+90:1.9) -- (-\x * 30+120:1.9);


\foreach \x  in { 2, 3,..., 6}
	\draw (q1) -- (-\x * 30+120:1.9);
\foreach \x  in { 6,7, ..., 10}
	\draw (q3) -- (-\x * 30+120:1.9);
\foreach \x  in {10, 11, 12, 1, 2}
	\draw (q2) -- (-\x * 30+120:1.9);

\end{tikzpicture}
\caption{Expressing $(|\a|/3)$-loop $x_0, x_1, \ldots, x_{11}, x_0$ as the boundary of a $2$-chain in $\sRips(X;a(r),2,b(r))$, we span a shaded triangle ($2$-simplex) on the three equidistant points $x_0, x_4, x_8$ to obtain the central triangle, and then fill in other (non-shaded) triangles. Note that all triangles are of diameter at most $|\a|/3$. The length of a filling of the central triangle is at most $|\a|$, while the lengths of fillings of other triangles is at most $2|\a|/3$.}
\label{FigNullhomotopy}
\end{figure}
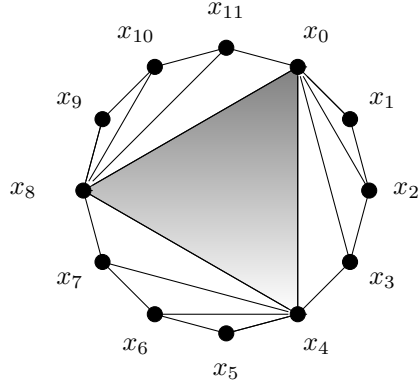

Together we obtain 
$$
\di \sum_{l=1}^{k_\alpha}  \sigma_l = \di \sum_{i=1}^k g_i \sum_{p=1}^{k_p} \tau_{i,p} + \di \sum_{j=1}^{k'} h_j \Delta_j
$$
and define a $2$-cycle
$$
C=\sum_{l=1}^{k_\alpha} \sigma_l - \sum_{i=1}^k  \sum_{p=1}^{k_p} g_{i}\tau_{i,p} -  \sum_{j=1}^{k'} h_j \Delta_j.
$$
Defining $Q_r=[C]$ we directly see that (1) holds as $C$ does not depend on parameter $r$ specifically but rather on its lower (unattained) bound $|\a|/3$.

We proceed by showing $Q_r$ is non-trivial using the winding number. By our assumptions and Lemma \ref{LemmaD} there exists a compact geodesically convex neighborhood $N$ of $\a$ in which $\a$ is the shortest non-contractible loop. By Lemma \ref{LemmaB} we may assume that each loop in $N$ is homotopic to the $k$-fold concatenation of $\a$ for some $k\in \ZZ$. Choose $T>0$ such that $N \supset \cN(\a, T)$ and $\nu=d(\cN(\a, T), N^C)>0$. Choose $\mu_T$ according to Lemma \ref{Lemma1} and $D$ according to Lemma \ref{LemmaDa}. Choose $B_1 \in (|\alpha|/3, \min\{\mu_T, D\}/3)$ so that $|\alpha|-2 B_1>0$, and define $B_2=\min \{|\alpha|-2 B_1, \nu\}$ (in the last two paragraphs of Remark \ref{Rem:brilliant}, this sentence is handled using condition $a < \nu$ to obtain the conclusion for the Rips complex). By Proposition \ref{PropMain1} the winding number $\o_N$ is an invariant of each single homology class in $H_2(\sRips(X;a(r),2,b(r));G) $. Obviously the winding number of the trivial class is $0$. On the other hand $\o_N(C) = \pm 1$ by the following argument:
\begin{itemize}
 \item All simplices of the form $\Delta_j$, which are contained in $N$, are of diameter less than $|\a|/3$ hence are of winding number $0$ by Remark \ref{Rem1}.
 \item Each non-central triangle (see Figure \ref{FigNullhomotopy}) of the form $\tau_{i,p}$ or $\s_l$, which is contained in $N$, has a filling of length  at most $2|\a|/3$ hence is of winding number $0$ by Lemma \ref{LemmaDa} and the bottleneck property.
 \item Suppose a central triangle $\Delta$ in $N$ is of the form $\tau_{i,p}$. As it is of diameter at most $|\a|/3$  it is of winding number $\pm 1$  or $0$ by Corollary \ref{Cor2}. Assume $\o_N(\Delta)=\pm 1$. By the bottleneck property $\diam(\Delta)=|\a|/3$ and by geodesic convexity of $N$ the entire $\b_i$ is contained in $N$. This implies that $\b$ equals $\a$ as a subset of $X$, which contradicts our assumptions on $\b_i$.
 \item The central triangle of the form $\s_l$ is of winding number $\pm 1$ as its filling is $\a$.
\end{itemize}

In order to prove (2) note that by Remark \ref{Rem1} and Lemma \ref{LemmaDa} no triple of diameter less than $|\a|/3$ has a non-trivial winding number $\o_N$.
\end{proof}

The following two corollaries recast our main result in the setting of persistent homology, i.e., in the case when the coefficients form a field. In both cases persistence diagrams and corresponding barcodes exist by compactness (see the q-tameness condition of Proposition 5.1 of \cite{Cha2} for details). Given positive $a<b$ and a field $\FF$, an \textbf{elementary interval} $\{\FF\}_{r\in (a,b \rangle}$ (where ``$\ \rangle$'' denotes either an open or a closed interval endpoint) is a collection of vector spaces $\{V_r\}_{r>0}$ defined as:
\begin{itemize}
 \item $V_r \cong \FF$ for $r\in (a,b \rangle$,
 \item $V_r =0$ for $r\notin (a,b \rangle$,
\end{itemize}
and bonding maps $V_r\to V_{r'}$ for all $r<r'$ being identities whenever possible (and trivial maps elsewhere).

While the following two results are expressed using $\sRips(X;r,2,s \cdot r)$ for simplicity, the results of course also hold for $\sRips(X;a(r),2,b(r))$ with appropriate functions $a(r)$ and $b(r)$, see Figure \ref{FigEnd1} and the corresponding example for such an occurence.

\begin{corollary}
 \label{Cor5}
 Let $X$ be a compact geodesic semi-locally simply connected space and let $\FF$ be a field. Assume $\a$ is a geodesic circle in $X$ satisfying the following properties:
\begin{enumerate}
 \item $\a$ is a bottleneck loop;
 \item $\a$ is homologous to a non-trivial $\FF$-combination of loops $\beta_1, \beta_2, \ldots, \beta_k$ of length at most $|a|$, none of which is homotopic to $\a$ or $\a^-$; 
 \item $\a$ has arbitrarily small geodesically convex neighborhood.
\end{enumerate}
Then there exists $S>0$ such that for each $s<S$ the persistent homology 
$$
\{H_2(\sRips(X;r,2,s \cdot r);G) \}_{r>0}
$$
contains as a direct summand an elementary interval of the form
$$
\{\FF\}_{r\in (|\a|/3, \delta_\alpha\rangle}
$$
for some $\delta_\alpha> |\a|/3$.
\end{corollary}

\begin{corollary}
 \label{Cor6}
 Let $X$ be a compact geodesic semi-locally simply connected space and let $\FF$ be a field. Assume $\a$ is a geodesic circle in $X$ satisfying the following properties:
\begin{enumerate}
 \item $\a$ is a bottleneck loop;
 \item $\a$ has arbitrarily small geodesically convex neighborhood.
\end{enumerate}
Then there exists $S>0$ such that for each $s<S$ the length of $\a$ is detected by the persistent homology $\{H_2(\sRips(X;r,2,s \cdot r);G) \}_{r>0}$ in the following way:
\begin{enumerate}
 \item if $\a$ is a member of a shortest homology base of $X$, then $|\a|/3$ is a closed endpoint of a one-dimensional bar;
 \item if $\a$ is not a member of a shortest homology base of $X$, then $|\a|/3$ is an open beginning of a two-dimensional bar.
\end{enumerate}
\end{corollary}

\begin{proof}
 Part (1) follows from \cite{ZV}, part (2) from Corollary \ref{Cor5}.
\end{proof}

\subsection{Localizing geodesic circles}

Throughout this subsection we assume the setting of Corollary  \ref{Cor6}.

While Corollary  \ref{Cor6} is clear about detecting the lengths of $\a$ in terms of birth or death times of persistence, localizing (determining the location of) a geodesic circle $\a$ in question is a bit more complicated. 
\begin{enumerate}
 \item Assume loop $\a$ is not homologous to a non-trivial $\FF$-combination of loops  $\beta_1, \beta_2, \ldots, \beta_k$ of length at most $|a|$, none of which is homotopic to $\a$ or $\a^-$. In this case $\a$ is a member of each lexicographically minimal homology base (see \cite{ZV} for details) by definition. By the results of \cite{ZV} there is a unique correpsponding $1$-dimensional persistent homology class with a closed death-time $|\a|/3$. This class is generated by any $(|\a|/3)$-sample of $\a$ and any of its nullhomologies in $\Rips(X;r)$ (and hence $\sRips(X;r,2,s \cdot r)$) for $r$ slightly larger than $|\a|/3$ contains a $2$-simplex of diameter at least $|\a|/3$ whose vertices form a circumventing triple in a small neighborhood of $\a$ (see the Localization Theorem in \cite{ZV}). All other $2$-simplices in such a nullhomology can be chosen so that the length of a filling of its vertices is less than, say,  $3 |\a|/4$, although at least the ones sharing an edge with the above-mentioned simplex will be of diameter about $|\a|/3$. In particular, this means that $\a$ can be well approximated by a filling of a critical $2$-simplex that ensures the triviality of the mentioned homology class. 
 
 In the case of closed Rips complexes $\overline\Rips(X;r)$, the $1$-dimensional persistent homology class corresponding to $\a$ will first become trivial at $r=|\a|/3$, where each of its nullhomologies will contain an equilateral $2$-simplex of diameter $|\a|/3$, whose corresponding filling of vertices is $\a$. In this setting $\a$ can be reconstructed explicitly. 
 
 \item Assume  $\a$ is homologous to a non-trivial $\FF$-combination of loops  $\beta_1, \beta_2, \ldots, \beta_k$ of length less than $|\a|$. By the construction in Theorem \ref{ThmMain1} all simplices of the corresponding $2$-dimensional homology class $Q_r$ can be chosen to have diameter at most $\max_i |\beta_i|/3$, except for the ones forming the nullhomotopy of $L_\a$, i.e., the ones appearing on Figure \ref{FigNullhomotopy}. Of the later ones, only one  triangle (central on Figure \ref{FigNullhomotopy}) is circumventing and a filling of its vertices is a good approximation of $\a$, while the other triangles can be chosen so that their fillings are of length no more than $3|\a|/4$. By the invariance of the winding number, each representative of $Q_r$ contains a (with respect to $\a$) circumventing $2$-simplex, whose filling of its vertices is a good approximation of $\a$.
 
 Again, if we considered closed selective Rips complexes as an obvious analogue to the selective Rips complexes introduced in this paper, the filling of the vertices of this circumventing $2$-simplex would be $\a$.
 
  \item Assume  $\a$ is homologous to a non-trivial $\FF$-combination of loops  $\beta_1, \beta_2, \ldots, \beta_k$ of length at most $|a|$, none of which is equal to $\a$ and $n$-many of which are of length $|\a|$ for some $n>0$.  Let us call these $n$-loops $\g_1, \g_2, \ldots, \g_n$. For the sake of simplicity let us suppose that this assumption does not hold for $n-1$. As each lexicographically minimal homology base consists of geodesic circles by the results of \cite{ZV}, we may assume all $\beta_i$ are geodesic circles. In this case, contrary to (2), the construction of Theorem \ref{ThmMain1} results in  $n+1$ ``central'' triangles: their fillings are approximations of $\a$ and $\gamma_i$. 
  
  In this case $\a$ is also a member of some lexicographically minimal bases and the same reconstruction technique as in (1) may also apply for certain nullhomologies of $\a$. In fact, at least $n$-many $1$-dimensional bars cease to exist at $|\a|/3$, with the reconstruction procedure of (1) yielding (depending on the chosen nullhomologies) at least $n$ of the $\a, \gamma_1, \ldots, \gamma_n$. 
\end{enumerate}

To summarize the discussion, the location of $\a$ may be determined or approximated by a filling of the vertices of the appropriate $2$-simplex corresponding to $\a$ in terms of Corollary  \ref{Cor6}. When using selective Rips filtration or when approximating the persistence of $X$ by computing persistence on a generic finite sample, arbitrarily precise approximations of $\a$ are obtained. When using the closed version of selective Rips filtration, loop $\a$ may be reconstructed directly. 

An issue that arises in the context is the choice of appropriate representative of homology (cases (1) and (3)) or nullhomology (cases (2) and (3)), (i.e., the one described by \cite{ZV} and Theorem \ref{ThmMain1}), from which we may select an appropriate ``critical'' simplex. Our results imply that these representatives may be obtained in the following way:
\begin{description}
 \item[i] try replacing each simplex of a representative by a chain consisting of simplices of diameter less than $|\a|/3$;
 \item[ii] try replacing the remaining simplices of diameter more that $|\a|/3$ by  a chain consisting of simplices, whose filling of vertices is of length less than $3|\a|/4$;
 \item[iii] the only remaining simplices, for which such a reduction may not be obtained, are the ones, whose fillings localize loops $\a$ and $\gamma_i$ as described in (1)-(3). Each such simplex can be replaced by  a chain consisting of simplices, whose filling of vertices is of length less than $|\a|$ and a single simplex, whose filling is of length $|\a|$ and equals $\a$ (or possibly $\gamma_i$ in the case of (3)). 
\end{description}

In the computational setting one would compute persistent homology of a sufficiently dense finite subset $S \subset X$ equipped with the subspace metric induced from the geodesic metric on $X$. In this case an algorithmic procedure to extract or localize an approximation of $\a$ is of interest. Steps \textbf{i.}-\textbf{iii.} translate to the following procedure on the filtration by selective Rips complexes:
\begin{description}
 \item[a] Sort simplices by diameter (this is essentially what filtrations by Rips complexes and selective Rips complexes do).
 \item[b] At each $r>0$, sort simplices of diameter $r$ according to the length of a filling of its vertices, and use this order when constructing the boundary matrix.
 \item [c] Compute persistent homology using the standard algorithm on the obtained order. 
\end{description}

During this computation, each $2$-simplex that appears as a critical simplex in the sense of (1)-(3) above (i.e., in the sense that it either kills a long $1$-dimensional bar or gives birth to an appropriate $2$-dimensional bar, under the conditions that $S$ is sufficiently dense) will be a simplex, which approximates $\a$ and potentially $\gamma_i$ in the sense of  (1)-(3) above.

\subsection{Example} 
\label{SubSectExample}
We now return to the example mentioned in the introduction. Consider a sphere, cut it along the parallel $P$ at height $h$ above equator and keep the lower part (Figure \ref{FigEnd}). We  denote the obtained lower part by $Z$. Equip $Z$ with the geodesic metric. Parallel $P$ is not a geodesic circle in the entire sphere but becomes one in $Z$. Also, $P$ is a bottleneck loop in $Z$ and has an arbitrarily small geodesically convex neighborhood, meaning it can always be detected with selective Rips complexes. However, if $P$ is close enough to the equator (i.e., so close that $|P|/3$ is larger than the distance from $P$ to the south pole) it can not be detected by Rips complexes: in that case $\Rips(Z, |P|/3)$ is a cone with apex in the south pole.

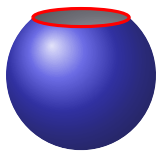
\begin{figure}[htbp]
\begin{center}
\begin{tikzpicture}[scale=.4]
    \begin{scope}
        \clip (130: 2.5) arc (180:-180: 1.6 and .3);
                \shade[ball color=blue!30!gray!20,shading angle=180] (0,0) circle (2.5);
    \end{scope}
    
    \begin{scope}
    \clip (130: 2.5) arc (-180:0: 1.6 and .3)
    arc (50:-230:2.5);
        \shade[ball color=blue!70!gray,opacity=0.90] (0,0) circle (2.5);
    \end{scope}
    to[out=200,in=-30] (130:2.5);
    \draw[red,very  thick] (130: 2.5) arc (180:-180: 1.6 and .3);

\end{tikzpicture}
\caption{A cut off sphere $Z$.}
\label{FigEnd}
\end{center}
\end{figure}

The computational examples below were produced using  Ripserer.jl software \cite{MC}.
 Space $Z$ was approximated by a random finite subset of 10 000 points. For computational convenience a density filter was applied that ensured that each pair of points is at distance at least $0.1$. Upon the remaining points a weighted graph was induced: each pair of points at distance at most $0.2$ was connected by an edge with weight being the Euclidean distance between points. The metric on the obtained collection of points is the geodesic metric induced by the described weighted graph and represents an approximation of $Z$ in a geodesic metric.

\begin{figure}[htbp]
\begin{center}
\includegraphics[scale=.4]{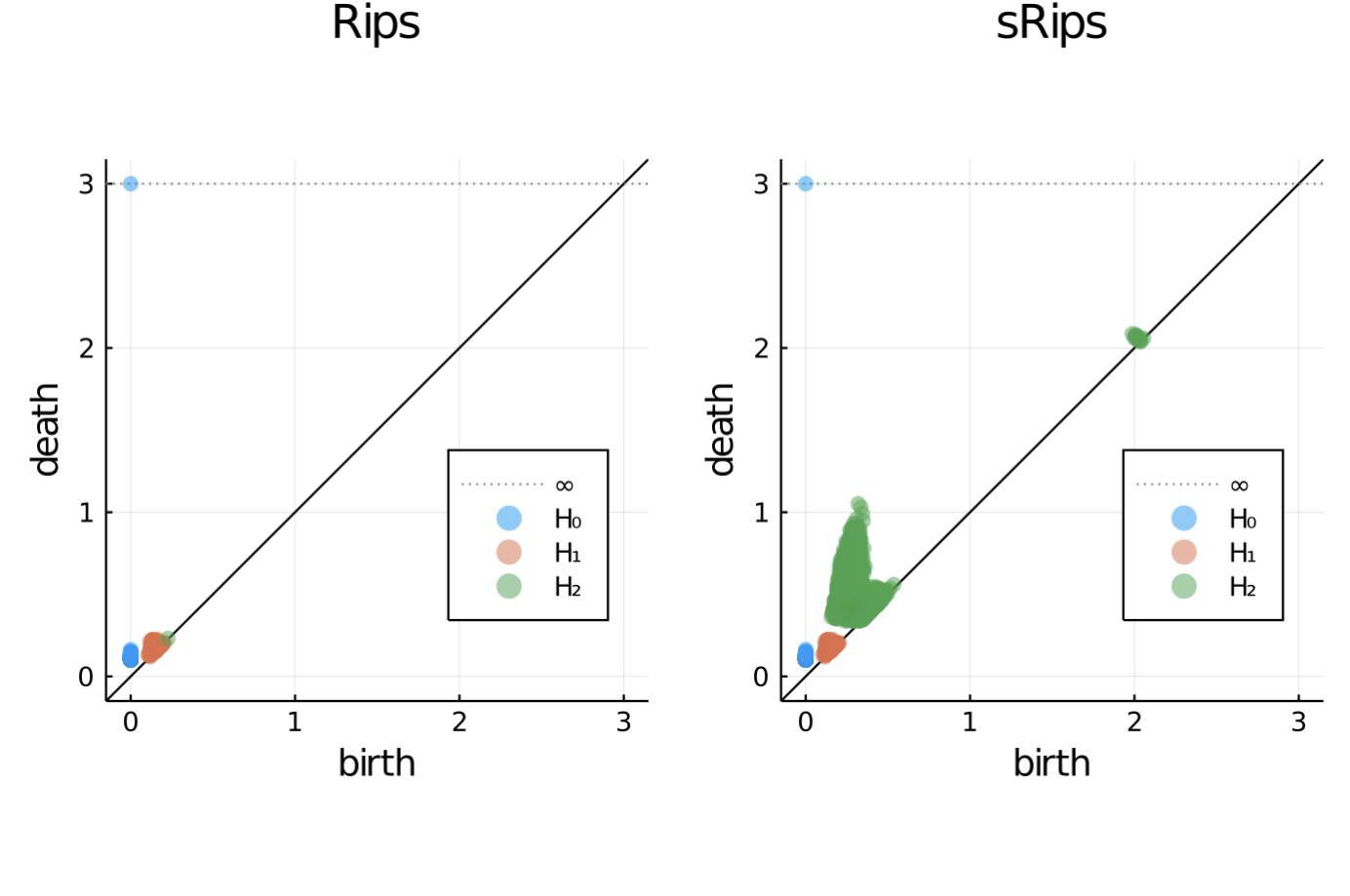}
\caption{Persistence diagrams of $Z$ using cutoff at $h=0.35$. The right diagram corresponds to $\sRips(Z; r,2, 0.3r)$. The circumference of $P$ equals approximately $2  \cdot \pi \cdot 0.94$ and about one third of it corresponds to the birth of $2$-dimensional points on the right. The diagram on the left using Rips standard complexes does not detect $P$. The lifespan of the last $2$-dimensional point appears to be short. However, it is stable with respect to various samples in the sense that it appears consistently.  Furthermore, its length $0.085$ for this specific sample is comparable to the scale of the difference of the radius of the equator and $P$, which is about $0.06$. A large number of $2$-dimensional points appearing up to about $0.6$ are a result of discretization and have a significant negative impact on the computational speed. See Figure \ref{FigEnd2} for a workaround. }
\label{FigEnd1}
\end{center}
\end{figure}

\begin{figure}[htbp]
\begin{center}
\includegraphics[scale=.4]{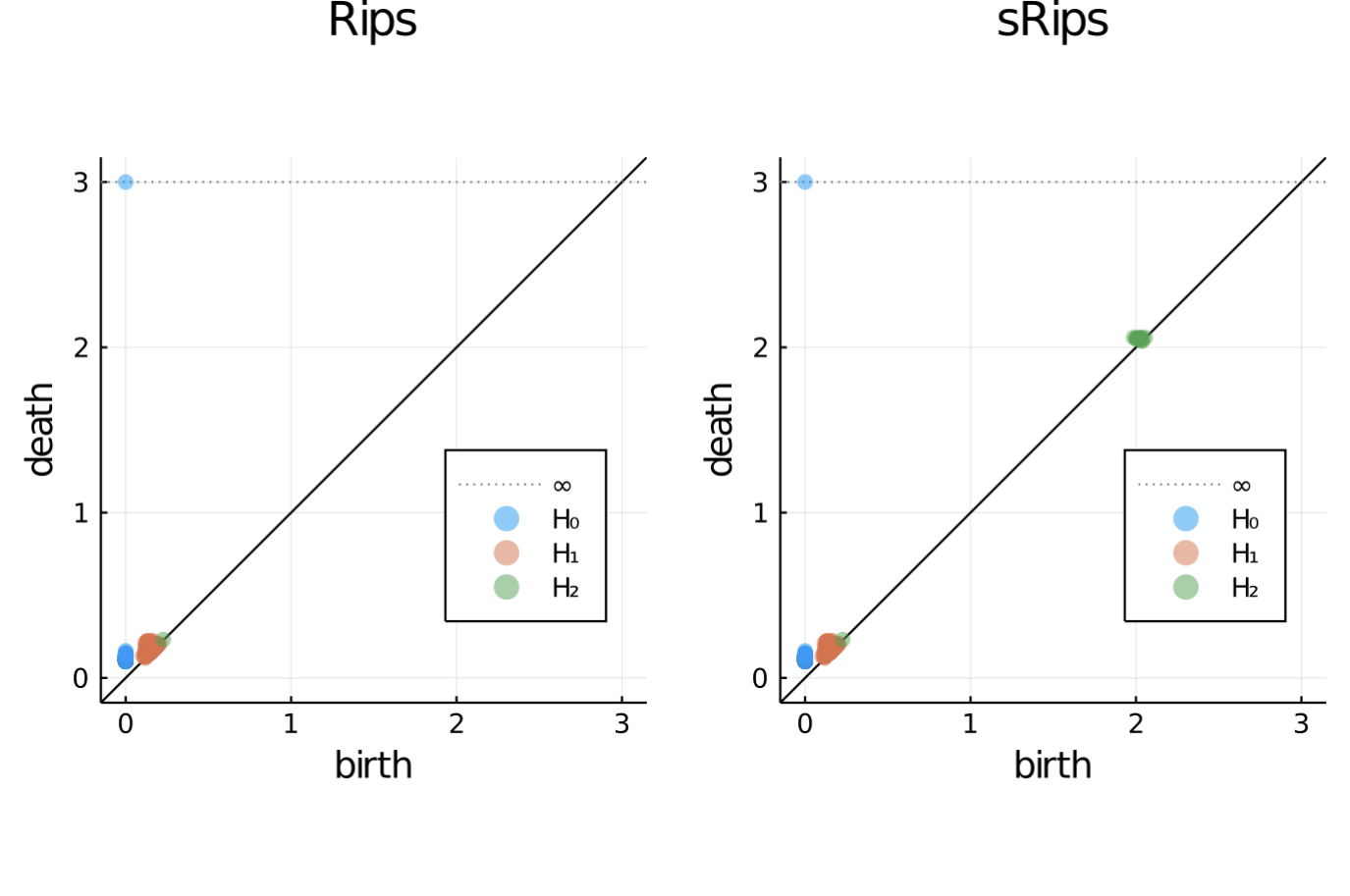}
\caption{Persistence diagrams of $Z$ using cutoff at $h=0.35$. The right diagram corresponds to $\sRips(Z; r,2, \min\{r, 0.7 + 0.3r\})$, meaning that at small scales we are in the region of classical Rips complexes, while for larger scales the effect of selective Rips complexes comes to the fore. With this modification we still detect $P$, reduce the initial noise and consequently speed up the process.}
\label{FigEnd2}
\end{center}
\end{figure}

\begin{figure}[htbp]
\begin{center}
\includegraphics[scale=.4]{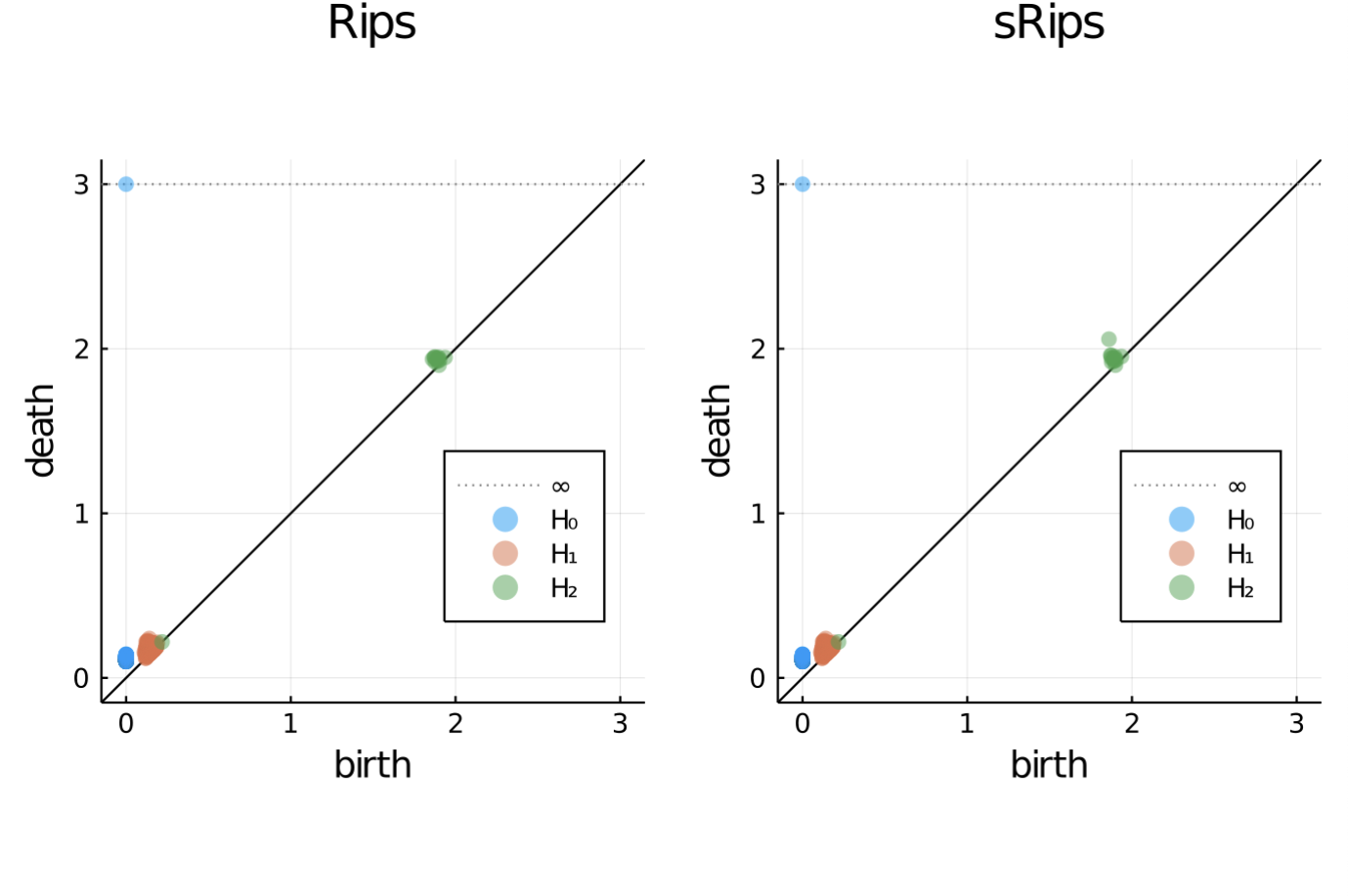}
\includegraphics[scale=.4]{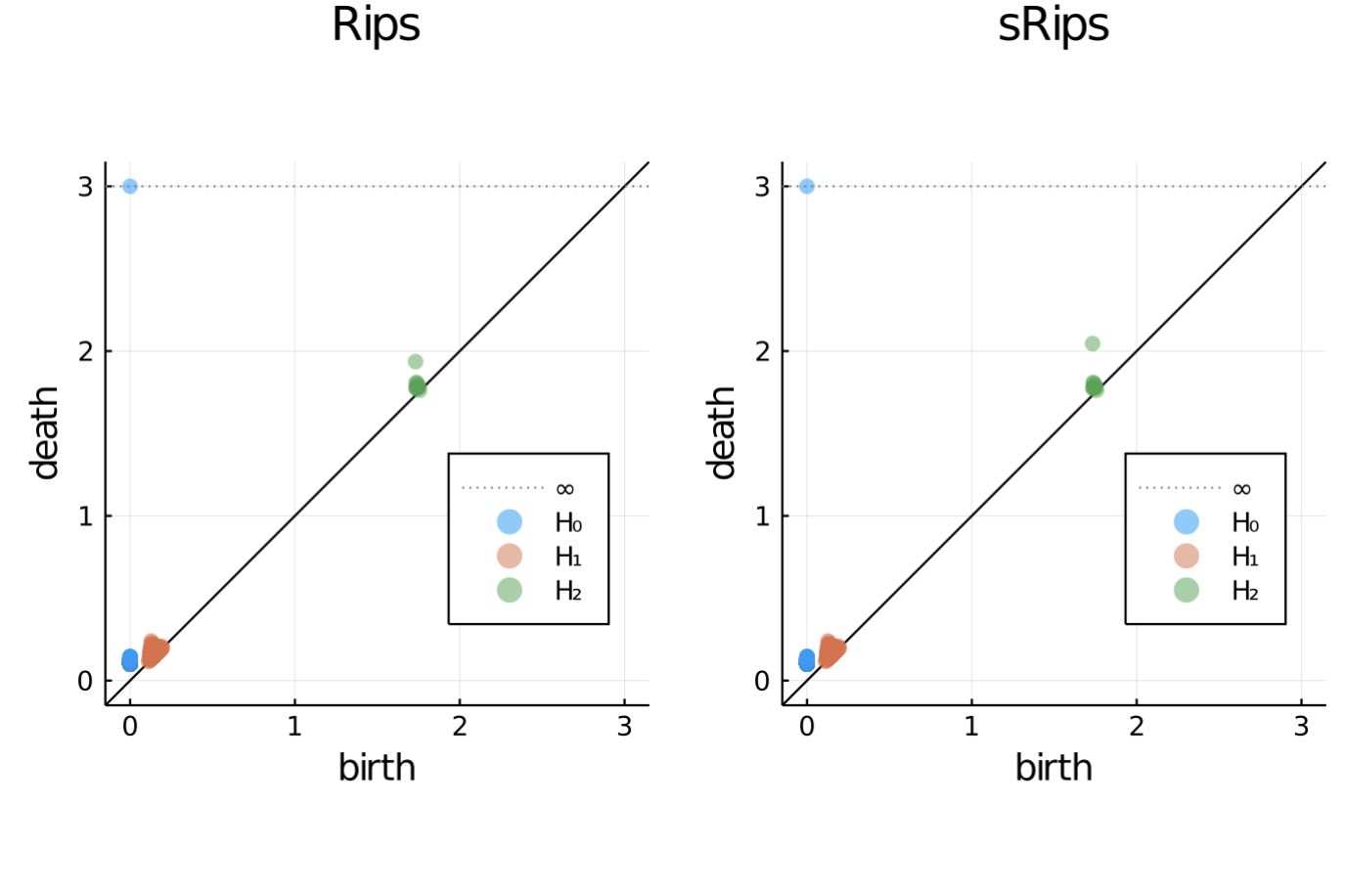}
\caption{Persistence diagrams of $Z$ using cutoff at $h=0.5$ (diagrams above) and $h=0.6$ (diagrams below). The right hand diagrams corresponds to $\sRips(Z; r,2, \min\{r, 0.9 + 0.1\cdot r\})$. Both Rips and selective Rips filtrations detect $P$ at this scale with the later inducing a longer-living $2$-dimensional point. In fact, the lifespan of the $2$-dimensional point in question can be made arbitrarily long with appropriate filtration functions.}
\label{FigEnd3}
\end{center}
\end{figure}


\end{document}